\documentclass[10pt, oneside]{article}
\usepackage{amsfonts, amsmath, amssymb, amscd, latexsym}
\usepackage{geometry}                
\usepackage{graphicx}
\usepackage{amssymb}
\usepackage{epstopdf}
\usepackage{hyperref}
\usepackage[utf8]{inputenc}
\usepackage{tikz}
\usepackage{mathtools}
\usepackage{array}
\usepackage{authblk}
\usepackage{calrsfs}
\usepackage[cal=boondoxo]{mathalfa}
\usepackage{amsthm}
\usepackage{multirow}
\usepackage{tabu}
\usepackage{diagbox}

\usepackage{comment}
\usepackage{breqn}
\usepackage{nth}

\newtheorem{thm}{Theorem}
\newtheorem{lem}[thm]{Lemma}
\newtheorem{cor}[thm]{Corollary}
\newtheorem{prop}[thm]{Proposition}

\theoremstyle{definition}
\newtheorem{definition}{Definition}

\theoremstyle{definition}
\newtheorem*{obs}{Observation}
                                                                                  
\newcommand{\Z}{\mathbb{Z}}

\newcommand{\N}{\mathbb{N}}
\newcommand{\Q}{\mathbb{Q}}
\newcommand{\K}{\mathbb{K}}
\newcommand{\eme}{\mathcal{m}}
\newcommand{\F}{\mathbb{F}}

\newcommand{\ene}{\mathcal{n}}

\author{Francisco Franco Munoz}
\date{}							
\title{\bf On minimal extensions of rings and applications}

\begin{document}

\newcommand{\Addresses}{{
  \bigskip
  \footnotesize

Francisco Franco Munoz, \textsc{Department of Mathematics, University of Washington,
    Seattle, WA 98195}\par\nopagebreak
  \textit{E-mail address}: \texttt{ffm1@uw.edu}
}}

\maketitle

\abstract{We study minimal extensions of local rings and their restriction to subrings. Some applications to subrings of $\K[x]/x^n$ and $\Z[x]/(p^N, x^n)$ are discussed.}


\section{Introduction}

The study of subrings of a given ring is a natural question, but one that has not been given enough attention in the literature. In \cite{AKKM} (and the references therein) the authors study the subrings of naturally occurring rings, such as $\Z^n$ and $\Z[x]/x^n$. \\
The purpose of this paper is to shed some light on subrings of a local rings and they way those can be built from extensions (precisely, minimal extensions). After some formal preliminaries we obtain our main result (Theorem \ref{main-thm}) that says we can relate subrings of $R$ with those of $S$ through a minimal extension $\varphi: R\to S$ in a very precise way, leading to an exact count. After establishing some basic result on (partial)-valuations, we are able to count the subrings of $\K[x]/x^n$, and obtain (section \ref{sec-4}) results giving precise estimates in the case of finite fields. At the end we briefly extend this discussion to subrings of $\Z[x]/(p^N, x^n)$ (section \ref{sec-5}) where similar counts are obtained.

\subsection{Notation}

All the rings considered are commutative and unital, and all the homomorphisms are unital. 

\noindent By a local ring $(R, \eme_R)$ we mean a ring with a unique maximal ideal $\eme_R$, not necessarily noetherian (sometimes called quasi-local rings). 

\noindent For a ring $R$ denote by $R^{\times}$ its group of units.

\subsection{On previous literature}

We have relied almost entirely on \cite{Franco} (which is self-contained). Some basic results of commutative algebra are assumed and we'll refer to $\cite{SP}$ for more extended discussions.


\section{Minimal extensions}

Let $(R,\eme_R), (S, \eme_S)$ be local rings. We record here few well known elementary results about surjective homomorphisms of local rings.

\begin{prop}\label{prop-1} Let $\varphi: R\to S$ be a surjective homomorphism of local rings. Then

\begin{enumerate}

\item $\eme_R= \varphi^{-1}(\eme_S)$, in other words, $\varphi$ is a \emph{local} homomorphism.

\item $\varphi(\eme_R)= \eme_S$.

\item $\varphi: R\to S$ induces an isomorphism on residue fields $R/\eme_R\cong S/\eme_S$.

\end{enumerate}

\end{prop}

\begin{proof}

\begin{enumerate}

\item Since $\varphi^{-1}(\eme_S)$ is a proper ideal it's contained in the maximal ideal so $\varphi^{-1}(\eme_S)\subseteq \eme_R$. So we need $\eme_R \subseteq \varphi^{-1}(\eme_S)$ equivalently $\varphi(\eme_R)\subseteq \eme_S$. Now, for a surjection of rings, the image of an ideal is an ideal. In the case of a local ring, $\varphi(\eme_R)$ is in fact a proper ideal: For if $\varphi(x)=1$ for some $x\in \eme_R$, then $1-x\in Ker(\varphi)\subseteq \eme_R$ so $1\in \eme_R$, a contradiction. Hence $\varphi(\eme_R)\subseteq \eme_S$.

\item From $\varphi^{-1}(\eme_S)\subseteq \eme_R$ get that $\eme_S = \varphi(\varphi^{-1}(\eme_S))\subseteq \varphi(\eme_R)$ since $\varphi$ is surjective, and by above $\varphi(\eme_R)\subseteq \eme_S$ hence equality: $\varphi(\eme_R)= \eme_S$.

\item We have $S/\eme_S\cong \varphi^{-1}(S)/\varphi^{-1}(\eme_S)= R/\eme_R$ by the above.

\end{enumerate}

\end{proof}

\begin{definition} A homomorphism $\varphi: R\to S$ is a minimal extension if it's surjective (in particular it's a \emph{local} homomorphism) and $I=Ker(\varphi)$ is a minimal nonzero ideal of $R$.
\end{definition}

\begin{lem}\label{lem-2} A minimal ideal $I$ satisfies $I\eme_R=0$ and so $I$ is a vector space over the residue field $R/\eme_R$. Moreover its dimension is one. Conversely, suppose that $J$ is an ideal of $R$ such that $J\eme_R=0$ and the dimension of $J$ over $R/\eme_R$ is one. Then $J$ is a minimal ideal. 
\end{lem}

\begin{proof} $I\eme_R\subseteq \eme_R$ and by minimality, $I\eme_R = 0$. The rest is clear.
\end{proof}

We need restriction of homomorphisms:

\begin{prop}\label{prop-3} Let $T\subseteq S$ a subring, and $\varphi: R\to S$ be a surjection of local rings, let $\bar{T}=\varphi^{-1}(T)$. Then $\bar{T}$ local implies $T$ local. Consider the following statements:

\begin{enumerate}

\item $\eme_{\bar{T}} = \eme_R\cap \bar{T}$ 

\item $\eme_{T} = \eme_S\cap T$

\item The composition $T\subseteq S\to S/\eme_S$ is surjective

\item The composition $\bar{T}\subseteq R\to R/\eme_R$ is surjective

\end{enumerate}

Then $1\iff 2$ and $3\iff 4$ and $3\implies 2$.\\

If $\varphi:R\to S$ is a minimal extension then  $\bar{T}$ local iff $T$ local. In that case, we have that $4 \iff (\varphi: \bar{T}\to T \text{ is a minimal extension})$
\end{prop}

\begin{proof} It's clear that $\bar{T}$ local implies $T$ local, since $\varphi$ is surjective. Assuming that, Proposition \ref{prop-1} is valid for $\varphi:R\to S$ and $\varphi:\bar{T} \to T$.

\begin{itemize}

\item $1\implies 2$: $\eme_T = \varphi(\eme_{\bar{T}}) = \varphi(\eme_R\cap \bar{T})\subseteq \varphi(\eme_R)\cap \varphi(\bar{T}) = \eme_S\cap T$ and the other inclusion $\eme_S\cap T\subseteq \eme_{T}$ is immediate. 

\item $2\implies 1$: $\eme_{\bar{T}} = \varphi^{-1}(\eme_T)=\varphi^{-1}(\eme_S\cap T)= \varphi^{-1}(\eme_S)\cap \varphi^{-1}(T) = \eme_R\cap \bar{T}$.

\item $3\iff 4$: $T\to S\to S/\eme_S$ surjective $\iff$ 

$T/(\eme_S\cap T)\cong S/\eme_S$ $\iff$ $\varphi^{-1}(T)/(\varphi^{-1}(\eme_S)\cap \varphi^{-1}(T))  \cong \varphi^{-1}(S)/\varphi^{-1}(\eme_S)$ $\iff$ $\bar{T}/(\eme_R\cap \bar{T})\cong R/\eme_R$ $\iff$ $\bar{T}\to R\to R/\eme_R$ surjective

\item $3 \implies 2$: $T\to S\to S/\eme_S$ surjective $\iff$ $T/(\eme_S\cap T)\cong S/\eme_S$ $\implies$ $\eme_S\cap T$ is a maximal ideal $\implies$ $\eme_S\cap T=\eme_T$.

\end{itemize}

Assume that $\varphi$ is a minimal extension. To show $T$ local implies $\bar{T}$ local. In fact, suppose that $\varphi^{-1}(\eme_T)\subseteq \ene$ where $\ene$ is a proper ideal of $\bar{T}$. Then $\eme_T = \varphi(\varphi^{-1}(\eme_T))\subseteq \varphi(\ene)\subseteq T$ since $\varphi$ is surjective. But also being surjective implies that $\varphi(\ene)$ is an ideal of $T$, and the latter is a local ring. If $\varphi(\ene)$ is not proper, then there's $x\in \ene$ such that $\varphi(x)=1$, so $x-1\in Ker(\varphi)$. But $Ker(\varphi)$ is square zero since it's a minimal ideal and so $(x-1)^2=x^2-2x+1=0$, so $1=x(2-x)$ and so $x\in \ene\subseteq \bar{T}$ is invertible, and so $\ene=\bar{T}$, contradiction. Hence $\varphi(\ene)\subseteq \eme_T$, i.e. $\ene\subseteq \varphi^{-1}(\eme_T)$, and $\varphi^{-1}(\eme_T)$ is the maximal ideal of $\bar{T}$.

Assume now $\bar{T}$ local. To show that $4 \iff$ ($\varphi: \bar{T}\to T$ is a minimal extension). 

Assume that $\bar{T}\to T$ is minimal. That means $I = Ker(\varphi)$ is a minimal nonzero ideal of $\bar{T}$, and so $\eme_{\bar{T}}I=0$. Let $x\in \eme_{\bar{T}}$. Then since $xI=0$, $[x]I=0$, where $[x]$ is the image of $x$ in $R/\eme_R$. But $I$ is a nonzero vector space (in fact one dimensional) over $R/\eme_R$, so $[x]=[0]$ i.e. $x\in \eme_R$, and so $\eme_{\bar{T}}\subseteq \eme_R\cap \bar{T}$, and the other inclusion holds since $\eme_R\cap \bar{T}$ is an ideal of $\bar{T}$, hence $\eme_{\bar{T}}=\eme_R\cap \bar{T}$. Moreover, we have 
$\bar{T}/(\eme_R\cap \bar{T})\subseteq R/\eme_R$ is an extension of fields, and the vector space $I$ has dimension one over both (since $\varphi:R\to S$ and $\varphi:\bar{T}\to T$ are minimal), and so the fields are equal $\bar{T}/(\eme_R\cap \bar{T})\cong R/\eme_R$. Conversely, if we have $\bar{T}\to R\to R/\eme_R$ surjective then $\bar{T}/(\eme_R\cap \bar{T})\cong R/\eme_R$ and so $\eme_{\bar{T}}=\eme_R\cap \bar{T}$ and the residue fields are the same. Hence the ideal $I$ is one dimensional over $\bar{T}/\eme_{\bar{T}}$ and the extension $\varphi: \bar{T}\to T$ is minimal by Lemma \ref{lem-2}.

\end{proof}

The use of this proposition will be in restricting local homomorphism to subrings and applying the next two theorems.

\begin{thm}\label{thm-4} Let $\varphi: R\to S$ be a minimal extension of local rings. 

\begin{enumerate}

\item Assume $R$ and $S$ contain a coefficient field $\K$ (the inclusions $\K\subseteq S$, $\K\subseteq R$ induce an isomorphism $R/\eme_R \cong S/\eme_S \cong \K$). If $A\subseteq R$ is a $\K$-subalgebra mapping onto $S$, then $A$ is local and $\eme_R^2\subseteq A$.

\item Assume that characteristic of $R$ is $p^N$ for some $N\geq 1$, the residue field of $S$ and $R$ is $\F_q$ and that $\eme_R$ is nilpotent. If $A\subseteq R$ is a subring that maps onto $S$, then $A$ is local and $\eme_R^2+pR \subseteq A$.

\end{enumerate}

\end{thm}

\begin{proof}

\begin{enumerate}

\item To show that $A$ is local, notice that the composition $\K\to A\to A/(\eme_R\cap A)\cong S/\eme_S = \K$ is a bijection, so $A=\K\oplus (\eme_R\cap A)$, which says that $A$ is local with maximal ideal $\eme_R\cap A=\eme_A$. And $\K$ is also a coefficient field for $A$. Now, the ideal $I = Ker(\varphi)$ is one dimensional over $R/\eme_R = A/\eme_A$, and so $I\cap A$ is an ideal of $A$ that's at most one dimensional over $A/\eme_A$. Hence either $I\subseteq A$ or $I\cap A = 0$. In the first case,  since $R/I \cong S$ and $A$ maps onto $S$, we have $A=R$ which doesn't map isomorphically onto $S$ since $I$ is nonzero. So, $I\cap A=0$ and the map $\varphi:A\to S$ is an isomorphism. Now as vector spaces $R=\eme_R\oplus \K$, and the same for $A=\eme_A\oplus \K$, and since the dimension of $I$ is one, the codimension of $\eme_A\subseteq \eme_R$ is one, and one can write $\eme_R=I\oplus \eme_A$ as vector spaces. Since $I\eme_R = 0$, we get $\eme_R^2 = \eme_A^2\subseteq A$, as claimed. 

\item To show that $A$ is local, notice that $A/(\eme_R\cap A)\cong S/\eme_S = \F_q$ so $\eme_R\cap A$ is a maximal ideal. Now, let $a\in A\setminus \eme_R\cap A$, then $a$ maps to a nonzero element of $\F_q$, hence there's $l$ such that $a^l-1\in Ker(\varphi)\subseteq \eme_R$, but $\eme_R$ is nilpotent so $(a^l-1)^m = 0$ for some $m$ which after expanding the equation gives that $a$ is invertible and it's inverse is in $A$. So $\eme_R\cap A=\eme_A$ is the unique maximal ideal of $A$, since its complement consists of invertible elements. Moreover $A$ also has residue field $\F_q$. Now, the ideal $I = Ker(\varphi)$ is one dimensional over $R/\eme_R = A/\eme_A$, and so $I\cap A$ is an ideal of $A$ that's at most one dimensional over $A/\eme_A$. Hence either $I\subseteq A$ or $I\cap A = 0$. In the first case,  since $R/I \cong S$ and $A$ maps onto $S$, we have $A=R$ which doesn't map isomorphically onto $S$ since $I$ is nonzero. So, $I\cap A=0$ and the map $\varphi:A\to S$ is an isomorphism. In particular, $A$ is a maximal subring with the same residue field as $R$. Indeed, a subring $A_1$ containing $A$ maps onto $S$ and the same arguments above give either that if $A_1\neq R$, then $A_1$ maps isomorphically onto $S$ and so $A_1=A$. Now, using ~\cite[Lemma 22]{Franco}, we obtain that if $A$ is a maximal subring with the same residue field as $R$, $A$ contains $\eme_R^2+pR$. This finishes the proof.
\end{enumerate}

\end{proof}

Let's record a consequence in the proof of the above theorem:

\begin{prop} With the same conditions as above:

\begin{enumerate}

\item If $A\subseteq R$ is a $\K$-subalgebra mapping onto $S$ not containing $Ker(\varphi)$, then $A\to S$ is an isomorphism and $A$ is a local, maximal $\K$-subalgebra of $R$.

\item If $A\subseteq R$ is a subring that maps onto $S$ not containing $Ker(\varphi)$ then $A\to S$ is an isomorphism and $A$ is a local, maximal subring of $R$ with the same residue field as $R$.

\end{enumerate} 

\end{prop}

\begin{obs} Throughout this paper we haven't made use of structural results such as Cohen's structure theorems (see \cite[tag/0323]{SP}) to simplify the hypothesis of the theorems. It's worthwhile to notice that for example, complete local rings of equal characteristic possess coefficient fields, hence the hypothesis in our theorems hold for a wide class of local rings. 
\end{obs}

Here's the main result that completes the analysis:

\begin{thm}\label{main-thm} Let $\varphi: R\to S$ be a minimal extension of local rings. Then one can describe the subrings $A$ mapping isomorphically to $S$ in the same cases as above:

\begin{enumerate}

\item Assume $R$ and $S$ contain a coefficient field $\K$. If $Ker(\varphi)\subseteq \eme_R^2$, then there are no $\K$-subalgebras of $R$ mapping isomorphically onto $S$ under $\varphi$. Otherwise, the set of such $\K$-subalgebras is naturally an affine space over $\K$ of dimension $\dim_{\K}(\eme_{S}/\eme_{S}^2)$. 

\item Assume that the characteristic of $S$ is $p^N$ where $N\geq 1$ and that the residue field of $S$ and $R$ is $\F_q$ and that $\eme_R$ is nilpotent. If $Ker(\varphi)\subseteq \eme_R^2+pR$, then there are no subrings of $R$ that map isomorphically onto $S$. Otherwise, the set of such subrings is naturally an affine space over $\F_q$ of dimension $\dim_{\F_q}(\eme_{S}/(pS+\eme_{S}^2))$.

\end{enumerate}

\end{thm} 

\begin{proof} 

\begin{enumerate}

\item By Theorem \ref{thm-4}, any $A$ mapping onto $S$ satisfies $\eme_R^2\subseteq A$. 

If $Ker(\varphi)\subseteq \eme_R^2$ then $A$ contains $Ker(\varphi)$ so $A=R$, which is not the case.

Assume now $Ker(\varphi)$ not contained in $\eme_R^2$, and fix a nonzero $z\in Ker(\varphi)$. Notice that $z$ is a generator as vector space over $R/\eme_R$. Define $V(R)=\eme_R/\eme_{R}^2$. To $A$, subalgebra mapping isomorphically onto $S$, let's assign the subspace $V(A) = \eme_A/\eme_R^2$. It's a $\K$-codimension one subspace in $V(R)$ that doesn't contain $\bar{z}$ (the image of $z$). Conversely, for $V\subseteq V(R)$ a $\K$-codimension one subspace, assign the subspace $A(V)=\K+ \tilde{V}+\eme^2$, where $\tilde{V}$ is a lift of $V$ to $\eme$. Notice that $A$ is closed under multiplication since $\tilde{V}$ contains $\eme^2$ and $\K \tilde{V}\subseteq \tilde{V}+\eme^2$. This provides with a 1-to-1 correspondence between the codimension one subspaces of $V(R)$ not containing $z$ and the $\K$-subalgebras $A$ mapping isomorphically onto $S$. Now, since $\eme_{R}=\eme_{A}\oplus \K z$, projection in the first component gives a bijection between $\{ V\subseteq \eme_{R}/\eme_{R}^2 \text{ codimension 1}\, |\ \bar{z}\notin V \}$ with $\eme_{A}/\eme_A^2\cong \eme_B/\eme_B^2$ which is the affine space sought after.

\item  By Theorem \ref{thm-4}, any $A$ mapping onto $S$ satisfies $\eme_R^2+pR\subseteq A$. 

If $Ker(\varphi)\subseteq \eme_R^2+pR$ then $A$ contains $Ker(\varphi)$ so $A=R$, which is not the case.

By the proposition above, $A$ is a maximal subring with the residue field as $R$ and $A$ doesn't contain $Ker(\varphi)$. Now, by \cite[Theorem 28]{Franco}, the maximal subrings $A$ with the same residue field as $R$ are in one to correspondence with the codimension one subspaces of $\eme_R/(\eme_R^2+pR)$, the bijection being $A\mapsto V(A)$ with the notation above. Fix $z\in Ker(\varphi)$ nonzero. Then $\bar{z}\in \eme_R/(\eme_R^2+pR)$ is nonzero, and notice that $\eme_A/(\eme_R^2+pR)$ is a codimension one subspace, that doesn't contain $\bar{z}$, and furthermore $\eme_A/(\eme_R^2+pR) \cong \eme_S/(\eme_S^2+pS)$. In this manner, as above, we get a bijection between the set of those $A$ mapping isomorphically onto $S$ and $\eme_S/(\eme_S^2+pS)$. 

\end{enumerate}

\end{proof}

We need conditions in which to apply the Proposition 3 in the study of subrings.

\begin{prop}\label{prop-7} Under the following conditions, given a minimal ring extension of local rings $\varphi: R\to S$ and a local subring, $T\subseteq S$, Theorem 6 applies to the restriction $\varphi^{-1}(T)\to T$:

\begin{enumerate}

\item All $R, S, T$ have the same coefficient field $\K$.

\item The characteristic of $R$ is $p^N$ for some prime $p$ and $N\geq 1$, $\eme_R$ is nilpotent and all $R, S, T$ share the same residue field, a finite field $\F_q$.

\end{enumerate}

\end{prop}

\begin{proof} By Proposition \ref{prop-3} the restriction $\varphi^{-1}(T)\to T$ to a local subring $T$ of a minimal extension is a minimal extension provided we have the composite map $T\to S\to S/\eme_S$ surjective. But this is the condition we are assuming in either case. So we only need to check the conditions of Theorem \ref{main-thm}.
\begin{enumerate}

\item Since $\K\subseteq T$, $\K\subseteq \varphi^{-1}(T)$, all the algebras involved have the same coefficient field $\K$, which are the conditions of Theorem \ref{main-thm} part 1.

\item Since $T$ has residue field $\F_q$, so does $\varphi^{-1}(T)$. The other conditions of Theorem \ref{main-thm} part 2 are satisfied.

\end{enumerate}

\end{proof}

As a result, we can compute dimensions and relate them with the existence of subalgebras. 

\begin{prop}\label{prop-8} Assume that we're in the coefficient field case of $\varphi:R\to S$ minimal extension and that $d(R) := \dim_{\K}(\eme_R/\eme_R^2)$ is finite. The following are equivalent:
\begin{enumerate}
\item $R$ possesses a subalgebra $A$ mapping isomorphically onto $B$ 
\item $Ker(\varphi)$ is not contained in $\eme_R^2$
\item $d(R)=d(S)+1$

\end{enumerate}
\end{prop}

\begin{proof}
\item $1\iff 2$ Theorem \ref{main-thm} part 1.
\item $2 \iff 3$ The map $\eme_R\to \eme_S$ gives an isomorphism $\eme_R/(\eme_R^2+Ker(\varphi))\cong  \eme_S/\eme_{S}^2$ and  $\displaystyle \eme_{R}/(\eme_{R}^2+Ker(\varphi))\cong \frac{\eme_{R}/\eme_{R}^2}{(\eme_{R}^2+Ker(\varphi))/\eme_{R}^2}$. From here, since $\dim_{\K}(Ker(\varphi))=1$, it's clear $Ker(\varphi)\not\subset \eme_R^2$ $\iff$ $\dim_{\K}((\eme_{R}^2+Ker(\varphi))/\eme_{R}^2)=1$ $\iff$ $d(R)=d(S)+1$.

\end{proof}


\section{Valuations}

\begin{definition} A (commutative) partial-monoid is a set $(M,+)$ endowed with a partial binary (commutative) operation that has a neutral element $0$ (i.e. $a+0$ is always defined and $a+0=a=0+a$), and it's associative when defined (i.e. $a+b$ and $(a+b)+c$ are defined iff $b+c$ and $a+(b+c)$ are defined, and if so, $(a+b)+c=a+(b+c)$). A sub-partial-monoid $N\subseteq M$ is a subset such that if $a, b\in N$ and $a+b$ is defined, then $a+b\in N$. From here on all partial-monoids are commutative.
\end{definition}

\begin{definition} An ordered partial-monoid is triple $(M,+,\leq)$ where $(M, +)$ is a partial-monoid and $\leq$ is a partial order that's compatible with the sum, i.e. given $a\leq c$ and $b\leq d$, if $c+d$ is defined then $a+b$ is defined and $a+b\leq c+d$. 
\end{definition}

For us, examples of interest are $[\mathbf{m}]=\{0, ..., m\}$ (finite interval of non-negative numbers) with (partial) addition, and products $\mathcal{M}_{n,N}=[\mathbf{n-1}]\times [\mathbf{N-1}]$. Both are ordered partial-monoids, the first with natural order, the second with lexicographic order: $(a,b)\leq (c,d)$ iff $a\leq c$ or $a=c$ and $b\leq d$. These are total orders.

Relevant to our study we need valuation-like functions, defined on commutative rings with values in partial-monoids.

\begin{definition} Let $(M,\ast), (N, \star)$ be partial-monoids. A partial-homomorphism is a (total) function $\phi: M\to N$ (i.e. everywhere defined) with the properties:
\begin{enumerate}
\item $\phi(0_M)=0_N$ 
\item For all $x,y\in M$, if both $x\ast y$ and $\phi(x)\star\phi(y)$ are defined, then $\phi(x\ast y) = \phi(x)\star\phi(y)$.

\noindent $\phi$ is called semi-strict if in addition:

\item For all $x,y\in M$, if  $\phi(x)\star\phi(y)$ is defined, then $x\ast y$ is defined and $\phi(x\ast y) = \phi(x)\star\phi(y)$.
\end{enumerate}
\end{definition}

A typical example of a semi-strict partial-homomorphism is the inclusion $N\subseteq M$ of a sub-partial-monoid.

\begin{definition} Let $R$ a commutative ring. A partial-valuation is partial-homomorphism $\nu: R\setminus \{0\}\to M$ that is surjective, whose target is an ordered partial-monoid $(M,+)$ (where $(R\setminus \{0\}, \cdot)$ is the multiplicative partial-monoid of $R$), that satisfies:
\vspace{1 mm}

\noindent \text{\emph{Non-Archimedean condition:}} If $x+y$ is defined (i.e. not zero), then $\nu(x+y)\geq \min\{ \nu(x),\nu(y)\}$ in the following sense: for any $\mu\in M$ such that $\mu\leq \nu(x)$ and $\mu\leq \nu(y)$, we have $\nu(x+y)\geq \mu$. 
\vspace{1 mm}

\noindent $\nu$ is called semi-strict if the partial-homomorphism $\nu:R\setminus \{0\}\to M$ is semi-strict. 
\vspace{1 mm}

\noindent $\nu$ is called strict if it's semi-strict, $M$ is totally ordered and the equality $\nu(x+y)= \min\{ \nu(x),\nu(y)\}$ holds when $\nu(x)\neq \nu(y)$.
\end{definition} 

\begin{obs} 

\begin{enumerate}

\item The Non-Archimedean condition is most easily stated when $M$ possesses infima over any finite subset. Then  $\min\{ \nu(x),\nu(y)\}$ is actually an element of $M$ and the condition simply reads that $\nu(x+y)\geq \min\{ \nu(x),\nu(y)\}$

\item When $M$ is totally ordered and $\nu$ is strict, then the Non-Archimedean condition is the familiar one from valuations on fields. 

\end{enumerate}

\end{obs}

\begin{prop}\label{prop-9} Let $\nu: R\setminus \{0\}\to M$ be a semi-strict partial-valuation. Then for a subring $S\subseteq R$, the image $\nu(S)$ is a sub-partial-monoid. In fact, for any sub-partial-monoid of the multiplicative partial-monoid $T\subseteq R\setminus \{0\}$, $\nu(T)$ is a sub-partial-monoid.
\end{prop}

\begin{proof} Immediate because of the extra condition.
\end{proof}

\begin{definition} A semi strict partial-valuation $\nu:R\setminus \{0\}\to M$ is monomial-like over $R_1$, a subring of $R$, if it satisfies the following property: for $x,y\in R\setminus\{0\}$ such that $\nu(x)=\nu(y)$, there is $u\in R_1^{\times}$ such that either $x-uy=0$ or $\nu(x-uy) > \nu(x)$.
\end{definition}

Here's an important structural result:

\begin{thm}\label{thm-10} Let $\nu:R\setminus \{0\}\to M$ be a partial-valuation on $R$, monomial-like over some subring $R_1$, where $M$ is a finite partial-monoid. Suppose that $a_1, .... , a_d$ generate $M$ as a partial-monoid (for every element $a\in M$ there are constants $\alpha_1, ... , \alpha_n\in \N$ such that the sum $\alpha_1a_1+\dots+\alpha_da_d$ is defined and equal to $a$). Let $r_i\in R$ be elements whose valuations are $\nu(r_i)=a_i$. Then $r_i$ generate $R$ as an algebra over the subring $R_1$.
\end{thm}

\begin{proof} Let $a\in M$ be a maximal element ($M$ is finite), and let $r\in R$ such that $\nu(r)=a$. Then $a=\alpha_1a_1+... \alpha_da_d$, and so the ``monomial" $\tilde{r}=r_1^{\alpha_1}... r_d^{\alpha_d}$ is nonzero since $\nu$ is semi-strict and $\nu(\tilde{r})=a$. There's $u\in R_1^{\times}$ such that $r-u\tilde{r}\neq 0$ then $\nu(r-u\tilde{r}) > \nu(r)$ which is not possible. Hence $r=u\tilde{r}$. By a standard reverse induction argument the result follows, since we're assuming $M$ is finite.
\end{proof}

A natural example: take any field $\K$ and consider the $\K$-algebra $\K[x]/x^n$, where the partial-valuation is $\nu(a_ix^i+ \text{higher order terms})=i$, taking values in $[\mathbf{n-1}]$. It's easily checked that this is a strict partial valuation. This valuation is monomial-like over the coefficient field $\K$:

\begin{lem}\label{lem-11} Let $a,b\in \K[x]/x^n$ with $\nu(a)=\nu(b)$, then there is a nonzero $u\in \K$ such that either $a=ub$ or $\nu(a-ub)$ has valuation strictly larger than $\nu(a)$.
\end{lem}

\begin{proof} Let $a=a_mx^m + ... $, $b=b_mx^m + .... $, with $a_m, b_m$ nonzero, then take $u=a_mb_m^{-1}$ and the result follows. 
\end{proof}

\begin{lem}\label{lem-12}  Define the following function on $R=\Z[x]/(p^N, x^n)$: Write a nonzero element $x$ as a sum of powers in increasing order $x=a_kx^k+ \text{ higher order terms}$, where $a_i\in \Z/p^N$ is nonzero, then set $\nu(x)=(k,\nu_1(a_k))$ (where $\nu_1$ is the natural partial-valuation on $\Z/p^N$ given by $\nu_1(up^m)=m$ where $u$ invertible). Then $\nu:R\setminus \{0\}\to \mathcal{M}_{n,N}$ is a strict partial-valuation, which is monomial-like over the coefficient ring $\Z/p^N$.
\end{lem}

\begin{proof} Let $z=a_jx^j+\text{ higher order terms}$, $w=b_kx^k+\text{higher order terms}$. Notice that $(j, \nu_1(a_j))+(k, \nu_1(b_k))$ is defined if and only if $j+k<n$, and $\nu_1(a_j)+\nu_1(b_k)<N$. So if this is the case, and since $\nu_1$ is a strict partial valuation (with values in $[\mathbf{N-1}]$), we have $a_jb_k\neq 0$, and $\nu_1(a_jb_k)=\nu_1(a_j)+\nu_1(b_k)$, so $zw= a_jb_kx^{j+k}+\text{higher order terms}$, and $\nu(zw)=(j+k, \nu_1(a_jb_k))=(j,\nu_1(a_j))+(k,\nu_1(b_k))=\nu(z)+\nu(w)$. This shows it's semi-strict. To show it's strict, notice that $\mathcal{M}_{n,N}$ is indeed totally ordered (with lexicographic order as described before), and moreover, when $\nu(z)\neq \nu(w)$, either $j\neq k$ or $\nu_1(a_j)\neq \nu(b_k)$. 

\begin{enumerate}

\item Say that $j\neq k$ and without loss of generality, $j < k$, then $z+w = a_jx^k+\text{higher order terms}$, and so $\nu(z+w)= \nu(z)=\min\{\nu(z), \nu(w)\}$ since by the definition of lexicographic order here $(j, \ast) < (k, \star)$ for any $\ast, \star$ when $j < k$.

\item Say that $j=k$, and without loss of generality, $\nu_1(a_j) < \nu_1(b_j)$. Then $z+w=(a_j+b_j)x^k+\text{higher order terms}$, and $\nu(z+w)=(j, \nu_1(a_j+b_j))=(j, \nu(a_j))=\min\{\nu(z), \nu(w)\}$ since $\nu_1$ is a strict partial valuation and using again the definition of lexicographic order.

\end{enumerate}

Finally notice by definition $\nu_1$ satisfies that $\nu_1(\alpha)=\nu_1(\beta)$, for $\alpha, \beta\in \Z/p^N$ implies there exist $u\in (\Z/p^N)^{\times}$ such that $\alpha=u\beta$. Hence if $z=\alpha x^m+\text{higher order terms}$, $w=\beta x^m+\text{higher order terms}$, and $\nu(z)=\nu(w)$, one has using the $u$ before that $z-uw = (\alpha-u\beta)x^m+ \text{higher order terms}$, has only powers higher than $m$, hence if nonzero, $\nu(z-uw) > \nu(z)=\nu(w)$.
\end{proof}


\section{Subalgebras of $\K[x]/x^n$}\label{sec-4}

\subsection{Setting}

Let $\K$ be a field and consider the $\K$-algebra $\K[x]/{x^n}$. Of course this is the same as $\K[[x]]/x^n$ and this viewpoint will be more appropriate later.

Let $R\subseteq \K[x]/x^n$ be a $\K$-subalgebra. Notice that for the prime fields $\Q ,\, \F_p$, $\K$-subalgebra is the same as a subring. All the linear maps, bases, and subalgebras are assumed to be $\K$-linear, unless otherwise specified.

\begin{definition} For a nonzero polynomial $r\in \K[x]/x^n$ one has a unique minimal $i$ such that $r=a_ix^i+...\text{ higher order terms}$, with $a_i\neq 0\in \K$. Define $\nu(r)=i$. This is the strict partial-valuation defined to above $r$. 

Here $i$ is called an exponent of $R$ and we define $E(R)$ as the set of exponents.
\end{definition}  

\begin{obs} Define $\nu(0)=\infty$ as a formal symbol and with the rule $i+\infty=\infty$ for any $i\in [0,n-1]$ and natural order $i\leq \infty$ for all $i$. 
\end{obs}

\begin{lem} 
\begin{enumerate}
\item The set $[0, n-1]\cup \{\infty\}$ is an ordered monoid. 
\item $\displaystyle \nu: \K[x]/x^n\to [0, n-1]\cup \{\infty\}$ is a homomorphism: $\nu(1)=0, \nu(r_1r_2)=\nu(r_1)+\nu(r_2)$. 
\item The Non-Archimedean property holds: for any two elements $r_1, r_2$, $\nu(r_1+r_2)\geq \min\{\nu(r_1),\nu(r_2)\}$ and equality holds if $\nu(r_1)\neq \nu(r_2)$.
\end{enumerate}
\end{lem}

\begin{proof} Immediate.
\end{proof}

\begin{prop} $E(R)$ is partial-monoid. 
\end{prop}

\begin{proof} By Proposition \ref{prop-9}. Equivalently, it's the image of the multiplicative monoid $\K[x]/x^n$ under a monoid homomorphism.
\end{proof}


\subsection{The set of exponents $E(R)$ and generators}

A subalgebra $R\subseteq \K[x]/x^n$ lifts to a subalgebra $\tilde{R}$ of $\K[[x]]$ of finite codimension.

\begin{lem} The finite codimension subalgebras of $\K[[x]]$ are exactly those coming from $\K[x]/x^n$.
\end{lem}

\begin{proof}
 By \cite[Theorem 1]{Franco}, those finite codimension subalgebras correspond to subalgebras of finite dimensional quotients $\K[[x]]/I$ where $I\subseteq \K[[x]]$ is an finite codimension ideal. Now, $\K[[x]]$ is a DVR and its nonzero ideals are $x^n\K[[x]]$, hence the claim. 
\end{proof}

The discrete valuation on $\K[[x]]$ is: $\nu(a_ix^i+...\text{ higher order terms}) = i$, the same as before. An important property of this valuation is (with proof as in Lemma \ref{lem-11}) that this is a monomial-like valuation over $\K$.

\begin{lem} We have $E(\tilde{R})=E(R)\cup \{n, n+1, n+2, ... \}$ hence $E(\tilde{R})$ is a numerical monoid, i.e. a submonoid of $(\N, +)$ with finite complement (It's also true that such a monoid has a unique, finite, set of minimal generators \cite{RS}).
\end{lem}

\begin{proof} Immediate.
\end{proof}

From now we'll work with finite index $\tilde{R}\subseteq \K[[x]]$ in this way.

\begin{prop} $\tilde{R}$ contains $x^n\K[[x]]$ iff $E(\tilde{R})$ contains $\{n, n+1, ...\}$.
\end{prop}

\begin{proof} The converse needs to be checked only. Assume $\{n, n+1, .... \}\subseteq E(\tilde{R})$. Then there are $r_{k}\in \tilde{R}$ monic such that $\nu(r_k)=k$, for $k\geq m$. Write $r_n = x^n+\alpha_{n+1}x^{n+1}+\alpha_{n+2}x^{n+2}+... $. There are constants $\beta_k$ such that the sequence $s_l = r_n-\beta_{n+1}r_{n+1}-.... -\beta_{n+l}r_{n+l}$ has coefficients $0$ for $x^m$ for $n < m < l$, hence $s_l$ converges to $x^m$. Since $\tilde{R}$ is complete (finite index in $\K[[x]]$), $x^m\in \tilde{R}$. 
\end{proof}

Consider now the map $\pi: \K[[x]]\to \K[[x]]/x^n$. Then for $\tilde{R}\subseteq \K[[x]]$ containing $x^n\K[[x]]$, let $R\subseteq \K[[x]]/x^n$ be its image. 

\begin{lem} $\displaystyle \eme_{R}/\eme_{R}^2 \cong \eme_{\tilde{R}}/(\eme_{\tilde{R}}^2+x^n\K[[x]])\cong \frac{\eme_{\tilde{R}}/\eme_{\tilde{R}}^2}{(\eme_{\tilde{R}}^2+x^n\K[[x]])/\eme_{\tilde{R}}^2}$
\end{lem}

\begin{proof} Clear since $x^n\K[[x]]\subseteq \tilde{R}$.
\end{proof}

The following result connects generators of the monoid $E(\tilde{R})$ and algebra generators of $\tilde{R}$:

\begin{prop}\label{prop-19} Let $\{a_1, ... , a_d\}$ be the minimal generating set for $E(\tilde{R})$. The following are equivalent:

\begin{enumerate}

\item Then the vector space $\eme_{\tilde{R}}/\eme_{\tilde{R}}^2$ has a basis $\{r_1 ,... , r_d\}$ where $\nu(r_i)=a_i$
\item For all $n$ such that $x^n\K[[x]]\subseteq \tilde{R}$, let $E_{(n)} = \{a_i \, | a_{i} < n\}$ be the a minimal generating set for $E(R=\tilde{R}/x^n)$ as partial-monoid. Then $\{r_i \, | a_i\in E_{(n)}\}$ is a basis for $\eme_{R}/\eme_{R}^2$.
\end{enumerate}

\end{prop}

\begin{proof} 

\begin{enumerate}

\item Let the vector space $\eme_{\tilde{R}}/\eme_{\tilde{R}}^2$ have basis $\{r_1 ,... , r_d\}$ where $\nu(r_i)=a_i$.  Then by Lemma 16, a basis for $\eme_{R}/\eme_{R}^2$ is obtained consisting of those $r_i$ such that $\nu(r_i) < n$, i.e, by $\{r_i \, | a_i\in E_{(n)}\}$.

\item Conversely, taking $n$ equal to one plus the maximum of the $a_i$, we have $x^n\K[[x]]\subseteq \eme_{\tilde{R}}^2$ and hence the isomorphism $\eme_{R}/\eme_{R}^2 \cong \eme_{\tilde{R}}/\eme_{\tilde{R}}^2$ gives the result.

\end{enumerate}

\end{proof}

Theorem \ref{thm-10} and its proof applied to $R=\tilde{R}/x^n$ says:

\begin{lem}\label{lem-20} If $\{a_1, ... , a_d\}$ generate $E(R)$ as partial-monoid. Let $r_1, ... , r_d$ be monic such that $\nu(r_i)=a_i$. Then $r_i$ generate $R$ as algebra. Furthermore, if $n-1\in E(R)$ is not a generator, then $x^{n-1}\in \eme_R^2$, more precisely, $x^{n-1}$ is a nontrivial (not just one factor) monomial $r_1^{\alpha_1}....r_d^{\alpha_d}=x^{n-1}$ where $n-1=\alpha_1a_1+ ... + \alpha_da_d$.
\end{lem}

\begin{prop}\label{prop-21} The valuation maps are compatible under injection and projection: Let $R\subset \K[x]/x^n$ be a $\K$-subalgebra.
\begin{itemize}
\item $x^{n-1}\in R \iff n-1\in E(R)$
\item Let $n\geq m$ and consider the projection map $\varphi: \K[x]/x^n\to \K[x]/x^m$ given by annihilating $x^m$. Then the valuation maps are identical where defined: for any nonzero $\bar{r}\in \varphi(R)$, $\nu(\bar{r})=\nu(r)$ for any preimage $r$ of $\bar{r}$.
\item In particular, if the projection gives an isomorphism $R\cong \varphi(R)$, the sets of exponents are identical.
\item In general $E(\varphi(R))=E(R)\cap [0, m-1]$
\item The cardinality of $E(R)$ is $\#(E(R))=\dim_{\K}(R)$.
\end{itemize}
\end{prop}

\begin{proof} Immediate by definition.
\end{proof}

\begin{prop} \label{prop-22} For $E\subseteq [\mathbf{N-1}]$ a sub-partial-monoid, let $d(E)$ be the cardinality of its (unique) minimal generating set. If $R\subseteq \K[x]/x^n$ has $E(R)=E$, then $\dim_{\K}(\eme_R/\eme_R^2)\leq d(E)$.
\end{prop}

\begin{proof} Combine Proposition \ref{prop-19} and Lemma \ref{lem-20}.
\end{proof}


\subsection{Failure of equality in $\dim_{\K}(\eme_R/\eme_R^2)\leq d(E)$}\label{sec-4.3}

Proposition \ref{prop-22} says that $\dim_{\K}(\eme_R/\eme_R^2)\leq d(E)$ and that's the most that one can assert. Here's a family of examples for which the equalities don't hold.
\vspace{2 mm}

The algebra generated by $\{1, a=x^6+x^9, b=x^7, c=x^8\}$ inside $\K[x]/x^{18}$ has elements the powers 
\begin{itemize}
\item $a=x^6+x^9$
\item $b=x^7$
\item $c=x^8$
\item $a^2=x^{12}+2x^{15}$
\item $ab=x^{13}+x^{16}$
\item $b^2=x^{14}$
\item $ac=x^{14}+x^{17}$
\item $bc=x^{15}$
\item $c^2=x^{16}$
\end{itemize}

Hence a linear basis is $\{1, x^6+x^9, x^7, x^8, x^{12}, x^{13}, x^{14}, x^{15}, x^{16}, x^{17}\}$, and $E=\{0, 6, 7, 8, 12, 13, 14, 15, 16, 17\}$ and the generators of $E$ are $\{6, 7, 8, 17\}$ so $17$ is a generator of $E$ while $x^{17} = ac-b^2$ which belongs to $\eme^2$. We have also $\dim_{\K}(\eme_R/\eme_R^2)=3 < 4=d(E)$. This example ($a=6$) can be generalized to an infinite family as follows:

\begin{prop} Let $a\geq 6$, and consider the following subalgebra of $\K[x]/x^n$ where $n=2a+6$, generated by $\{1, a=x^a+x^{a+3}, b=x^{a+1}, c=x^{a+2}\}$. We have:
\begin{itemize}
\item $a=x^a+x^{a+3}$
\item $b=x^{a+1}$
\item $c=x^{a+2}$
\item $a^2=x^{2a}+2x^{2a+3}$
\item $ab=x^{2a+1}+x^{2a+4}$
\item $b^2=x^{2a+2}$
\item $ac=x^{2a+2}+x^{2a+5}$
\item $bc=x^{2a+3}$
\item $c^2=x^{2a+4}$
\end{itemize}

A linear basis is $\{x^a+x^{a+3}, x^{a+1}, x^{a+2}, x^{2a}, x^{2a+1}, x^{2a+2}, x^{2a+3}, x^{2a+4}, x^{2a+5}\}$, and $E=\{0, a, a+1, a+2, 2a, 2a+1, 2a+2, 2a+3, 2a+4, 2a+5\}$ and the generators of $E$ are $\{a, a+1, a+2, 2a+5\}$ so $2a+5$ is a generator of $E$ while $x^{2a+5} = ac-b^2$ which belongs to $\eme^2$. Further, $\dim_{\K}(\eme_R/\eme_R^2)=3 < 4=d(E)$.
\end{prop}

\begin{proof} The only need to check is the assertion regarding $E(R)$. But this follows from the assumption $a\geq 6$ which guarantees that the sum of any three nonzero elements of $E$ is larger than $2a+5$ (in fact, the minimum of the sum of any three nonzero elements is $3a > 2a+5$) and the sum of two nonzero elements is an element of the set $\{2a, 2a+1, 2a+2, 2a+3, 2a+4\}$, so indeed $2a+5$ is not the sum of two nonzero elements, so it's a generator.
\end{proof}


\subsection{Subalgebras of given shape and counting}

\begin{lem} The extension $\varphi: \K[x]/x^{n+1}\to \K[x]/x^{n}$ is minimal. Furthermore, for any $R$ a subalgebra of $\K[x]/x^{n+1}$ containing $x^n$, the extension $R\to R/x^n$ is minimal.
\end{lem}

\begin{proof} In fact $x^n\K[x]/x^{n+1}$ is the unique minimal ideal of $\K[x]/x^{n+1}$. The second follows as well (by applying restrictions of minimal extensions, Proposition \ref{prop-7}).
\end{proof}

Here's the main use of our results on minimal extensions:

\begin{thm}\label{thm-25} Let $R\subseteq \K[x]/x^{n+1}$ subalgebra.

\begin{enumerate}
\item If $x^{n}\in \eme_R^2$, there are no subalgebras mapping isomorphically onto $\varphi(R)=S$.
\item Otherwise, the set of such subalgebras is parametrized by an affine space of dimension $\dim_{\K}(\eme_S/\eme_S^2)$.
\end{enumerate}
\end{thm}

\begin{proof} This is Theorem \ref{main-thm} part 1, since all the subalgebras involved have coefficient field $\K$.
\end{proof}

\begin{cor} \label{cor-26} Let $\K=\F_q$ a finite field of $q$ elements. Let $R\subseteq \K[x]/x^{n+1}$ subalgebra, and $n\in E(R)$.
\begin{enumerate}
\item If $x^{n}\in \eme_R^2$, there are no subalgebras mapping isomorphically onto $\varphi(R)=S$.
\item Otherwise, the number of such subalgebras is $q^{d(S)}$, where $d(S)=\dim_{\K}(\eme_S/\eme_S^2)$. 
\end{enumerate}
\end{cor}

\begin{proof} Immediate from Theorem \ref{thm-25}, and Proposition \ref{prop-21} ($n\in E(R)$ iff $x^{n}\in R$).
\end{proof}

\begin{cor}\label{cor-27} Suppose that $A\subseteq \F_q[x]/x^{n+1}$ and that $n\notin E(A)$. Then the number of subalgebras of $\F_q[x]/x^{n+1}$ mapping isomorphically onto $\varphi(A)$ is $q^{d(A)}$ where $d(A)\leq d(E(A))$.
\end{cor}

\begin{proof} Let $S=\varphi(A), R=\varphi^{-1}(A)$. Since $n\notin E(A)$, the map $A\to S$ is an isomorphism and we're in situation $2$ of Corollary \ref{cor-26}, and so the number of such subalgebras is $d(S)=\dim_{\F_q}(\eme_S/\eme_S^2)=\dim_{\F_q}(\eme_A/\eme_A^2)=d(A)$ and $d(A)\leq d(E(A))$.
\end{proof}


\subsubsection{Shape}

\begin{definition}  Let $E\subseteq [\mathbf{n-1}]$ a sub-partial-monoid. The collection of subalgebras of $\K[x]/x^n$ of shape $E$ is the set $\mathcal{S}_n(E)=\{ R\subseteq \K[x]/x^{n} \, \mid E(R)=E\}$. This is a partition of the set of subalgebras of $\K[x]/x^n$.
\end{definition}

\begin{prop}\label{prop-28} Let $E\subseteq [\mathbf{n}]$ a sub-partial-monoid, $R\subseteq \K[x]/x^{n+1}$ subalgebra, $E(R)=E$.
\begin{enumerate}
\item If $n\in E$, the mapping $R\mapsto R/x^{n}$ induces a bijection of sets  $\mathcal{S}_{n+1}(E)\mapsto  \mathcal{S}_{n}(E\setminus \{n\})$
\item If $n\notin E$, the mapping $R\mapsto R/((x^n)\cap R)\cong R$, induces a mapping $\mathcal{S}_{n+1}(E)\mapsto  \mathcal{S}_{n}(E)$. Moreover, for $E_1\neq E_2$ sub-partial-monoids of $[\mathbf{n}]$ such that $n\notin E_1$, $n\notin E_2$, the images of $\mathcal{S}_{n+1}(E_1)$ and 
$\mathcal{S}_{n+1}(E_2)$ are disjoint, under the mapping just described.
\item If $n\notin E$, the mapping above $\mathcal{S}_{n+1}(E)\mapsto  \mathcal{S}_{n}(E)$ is surjective $\iff$ for all $B\in \mathcal{S}_{n}(E)$ the ring $R=\varphi^{-1}(B)$ possesses a subring $A$ mapping isomorphically to $B$ $\iff$ the kernel $Ker(R\to B)$ doesn't lie in $\eme_R^2$ $\iff$ $d(R)=d(B)+1$.

\end{enumerate}
\end{prop}

\begin{proof} 
\begin{enumerate}
\item By Proposition \ref{prop-21}, the mapping is well defined, and so is the inverse mapping $B\mapsto \varphi^{-1}(B)$, $B\subseteq \K[x]/x^n$. It's immediate to see that the composition in both ways yields the identity on the sets $\mathcal{S}$.
\item By Proposition \ref{prop-21}, for $n\notin E$, the the mapping $R\mapsto R/((x^n)\cap R)\cong R$ induces an equality of sets $E(R)=E(R/((x^n)\cap R))$, hence both claims follow. 
\item The last two equivalences are the content of Proposition \ref{prop-8}. For the first one, to have a surjective map $\mathcal{S}_{n+1}(E)\mapsto  \mathcal{S}_{n}(E)$ amounts to have that $R$ is not the only ring mapping to $B$, which by theory of minimal extensions (Theorem \ref{main-thm}), is the same as saying $R$ has a subring mapping isomorphically onto $B$.
\end{enumerate}
\end{proof}

In the setting of a finite field $\K=\F_q$, we can make an estimate of the number of subalgebras. For $E\subseteq [\mathbf{n-1}]$ a sub-partial-monoid, let $e(E)$ be the following quantity recursively defined, with $d(E)$ being the minimal number of generators (so $d(\{0\})=0$)

  \[ e_{n}(E) = \begin{cases}
              0 &\text{if~} n=1 \\
               e_{n-1}(E\setminus \{n-1\})     & \text{if~} n-1\in E, n > 1\\
               d(E)+e_{n-1}(E)     & \text{if~} n-1\notin E, n > 1\\
            \end{cases} \]

\begin{thm} Let $E\subseteq [\mathbf{n-1}]$. Then $\#(\mathcal{S}_{n}(E))$ is at most $q^{e_{n}(E)}$.
\end{thm}

\begin{proof} By induction. For $n=1$, the only such sub-partial-monoid is $E=\{0\}$ and $\mathcal{S}_{1}(E)=\{\F_q\}$ and obviously $\#(\mathcal{S}_{1}(E))=1=q^0=e_{1}(E)$.

Assume for $n$. And let consider $R\subseteq \F_q[x]/x^{n+1}$. Then $x^{n}\in R$ iff $n\in E(R)$.

\begin{itemize} 
\item $n\in E(R)$. By Proposition \ref{prop-28} part1, $\#(\mathcal{S}_{n+1}(E)) = \#(\mathcal{S}_{n}(E\setminus \{n\}))$ which is at most $q^{e_{n}(E\setminus \{n\})=e_{n+1}(E)}$.
\item $n\notin E(R)$. By Proposition \ref{prop-28} part 2, there are only two sets $F= F_1, F_2$ such that $\mathcal{S}_{n+1}(F)$ maps to $\mathcal{S}_{n}(E)$, namely, $F_1=E, F_2=E\cup\{n\}$. For each $B\in \mathcal{S}_{n}(E)$, the rings $A\in \mathcal{S}_{n+1}(E)$ mapping to $B$ give isomorphisms $A\cong B$, and every such $A$ is contained in $\varphi^{-1}(B)$, and by Corollary \ref{cor-27} there are $q^{d(B)}\leq q^{d(E)}$ of them, so the fiber has at most $q^{d(E)}$ elements, hence $\#(\mathcal{S}_{n+1}(E))\leq q^{d(E)}\#(\mathcal{S}_{n}(E))$ which by induction is $\leq q^{d(E)}q^{e_{n}(E)}=q^{d(E)+e_{n}(E)}=q^{e_{n+1}(E)}$ by definition.
\end{itemize}
\end{proof}

There's a way to make these estimates more precise with a finer partition of the set of algebras, which won't be pursued here. However, as a corollary of the proof, we have equalities in the following cases:

\begin{prop} For $E\subseteq [\mathbf{n}]$ sub-partial-monoid:
\begin{itemize}
\item If $n\in E$, $\#(\mathcal{S}_{n+1}(E)) = \#(\mathcal{S}_{n}(E\setminus \{n\}))$
\item If $n\notin E$, $\#(\mathcal{S}_{n+1}(E))\leq q^{d(E)}\#(\mathcal{S}_{n}(E))$ and furthermore, if for all $B\in \mathcal{S}_{n}(E)$, $d(\varphi^{-1}(B))=d(B)+1$, then $\#(\mathcal{S}_{n+1}(E))= q^{d(E)}\#(\mathcal{S}_{n}(E))$
\end{itemize}
\end{prop}

\begin{proof}
\begin{itemize}
\item Immediate.
\item If $n\notin E$, the above proof produces the inequality and so the only thing to check is the equality: $\#(\mathcal{S}_{n+1}(E))= q^{d(E)}\#(\mathcal{S}_{n}(E))$ provided for all $B\in \mathcal{S}_{n}(E)$, $d(B)=d(E)$. Indeed, if that's the case, as in the proof above, the fiber of the map for all 
$\mathcal{S}_{n+1}(E)\mapsto \mathcal{S}_{n}(E)$ has exactly $q^{d(B)}=q^{d(E)}$ elements. Furthermore, it's surjective by Proposition \ref{prop-28} part 3. Hence the equality $\#(\mathcal{S}_{n+1}(E))= q^{d(E)}\#(\mathcal{S}_{n}(E))$.
\end{itemize}
\end{proof}

\begin{prop} Let $B\subseteq \K[x]/x^n$ subalgebra and $R=\varphi^{-1}(B)$. If $d(R)=d(E(R))$, then $d(B)=d(E(B))$.
\end{prop}

\begin{proof}  We have two cases: $d(R)=d(B)$ or $d(R)=d(B)+1$, and in both cases, from Proposition \ref{prop-21}, $E(R)=E(B)\cup \{n\}$.
\begin{itemize}
\item $d(R)=d(B)$. We need to show that $d(E(R))=d(E(B))$. It's clear that $d(E(B))\leq d(E(R))$ and $d(E(R))=d(R)=d(B)\leq d(E(B))$ by Proposition \ref{prop-22}, hence equality $d(E(B))=d(B)$.
\item $d(R)=d(B)+1$. We need to show that $d(E(R))=d(E(B))+1$, which is equivalent to show that $n$ is a generator of $E(R)$. But if not, by Lemma \ref{lem-20}, then $x^n$ would belong to $\eme_R^2$ which is not the case since the condition $d(R)=d(B)+1$ as we have above, says that $Ker(\varphi)=(x^n)$ is not in $\eme_R^2$. This concludes the result.
\end{itemize}
\end{proof}


\begin{prop} With the same setup as above, assume $d(R)=d(B)+1$. If $d(B)=d(E(B))$ then $d(E(R))=d(R)$.
\end{prop}

\begin{proof} To show $d(E(R))=d(E(B))+1$, and the proof is the same as above.
\end{proof}

When $d(R)=d(B)$ the above is not true as section \ref{sec-4.3} on $\dim_{\K}(\eme_R/\eme_R^2)\leq d(E)$ shows.


\section{Subrings of $\Z[x]/(p^N, x^n)$}\label{sec-5}

\subsection{Setting}

To study the subrings of $R=\Z[x]/(p^N, x^n)$ in the framework of minimal extensions we need a consider a slightly larger family of rings. Most proofs in this section are analogous to those of $\K[x]/x^n$ and will be omitted for the most part.

\begin{definition} Let $n\geq 2$, $1\leq k\leq N$. $R_{n, N, k}$ is defined as the ring $R_{n, N, k}=\Z[x]/(p^N, x^n, p^kx^{n-1})$.
\end{definition}

\noindent This family ``interpolates" between family $\Z[x]/(p^N, x^n)$ in the sense that $R_{n, N, N}=\Z[x]/(p^N, x^n)$ and $R_{n, N, 0}=\Z[x]/(p^N, x^{n-1})=R_{n-1, N, N}$. Notice that the rings $R_{n, N, k}$ decrease as $k$ decreases from $N$ to $0$, more precisely:

\begin{lem} $R_{n, N, j}$ is a quotient of $R_{n, N, k}$ for $N \geq k\geq j$.
\end{lem}

We can extend the definition of the partial-valuation of Lemma \ref{lem-12} to this entire family:

\begin{definition} Let $R=R_{n, N, k}$. Write a nonzero element $x$ as a sum of powers in increasing order $x=a_kx^k+ \text{ higher order terms}$, where $a_i\in \Z/p^N$ is nonzero, let $\nu$ be the function $\nu: R_{n,N,k}\to \mathcal{M}_{n,N}$ defined by $\nu(x)=(k,\nu_1(a_k))$ (where $\nu_1$ is the natural partial-valuation on $\Z/p^N$ given by $\nu_1(up^m)=m$ where $u$ invertible). Here $(k, \nu_1(a_k))$ is called an exponent of $R$ and we define $D(R)$ as the set of exponents.
\end{definition}

\begin{lem} 
\begin{enumerate}
\item The set $\mathcal{M}_{n,N}=[\mathbf{n-1}]\times [\mathbf{N-1}]$ is a totally ordered monoid with the lexicographic order.
\item $\nu$ is a strict partial valuation on $R_{n, N, k}$ with values in the monoid $\mathcal{M}_{n,N}$ (whose image is the set where $(a,b)\in D(R)$ and $a=n$ then $b\leq k$). 
\item The Non-Archimedean property holds: for any two elements $r_1, r_2$, $\nu(r_1+r_2)\geq \min\{\nu(r_1),\nu(r_2)\}$ and equality holds if $\nu(r_1)\neq \nu(r_2)$.
\item $D(R)$ is a partial-monoid.
\end{enumerate}
\end{lem}


\subsection{The set of exponents $D(R)$ and generators}

Let $R\subseteq R_{n, N, k}$ a subring. Theorem \ref{thm-10} and its proof give:

\begin{lem} If $\{a_1, ... , a_d\}$ generate $D(R)$ as partial-monoid. Let $r_1, ... , r_d$ be monic such that $\nu(r_i)=a_i$. Then $r_i$ generate $R$ as algebra. Furthermore, if $(n-1, k-1)\in D(R)$ is not a generator, then $p^{k-1}x^{n-1}\in \eme_R^2$, more precisely, $p^{k-1}x^{n-1}$ is a nontrivial (not just one factor) monomial $r_1^{\alpha_1}....r_d^{\alpha_d}=p^{k-1}x^{n-1}$ where $n-1=\alpha_1a_1+ ... + \alpha_da_d$.
\end{lem}

\begin{prop}\label{prop-36} The valuation maps are compatible under injection and projection: Let $R\subseteq R_{n, N, k}$ be a subring:
\begin{itemize}
\item $p^{k-1}x^{n-1}\in R \iff (n-1, k-1)\in D(R)$
\item Let $k= j+1$ and consider the projection map $\varphi: R_{n,N,k}\to R_{n,N,j}$. Then the valuation maps are identical where defined: for any nonzero $\bar{r}\in \varphi(R_{n,N,k})$, $\nu(\bar{r})=\nu(r)$ for any preimage $r$ of $\bar{r}$.
\item In particular, if the projection gives an isomorphism $R\cong \varphi(R)$, the sets of exponents are identical.
\item In general $D(\varphi(R))=D(R)\cap \nu(R_{n,N,k})$
\item The cardinality of $R$ is $\#(R)=p^{\#(D(R))}$.
\end{itemize}
\end{prop}

\begin{prop} For $D\subseteq \mathcal{M}_{n,N}$ a sub-partial-monoid, let $d(D)$ be the cardinality of its (unique) minimal generating set. If $R\subseteq R_{n,N,k}$ has $D(R)=D$, then $\dim_{\F_p}(\eme_R/(\eme_R^2+pR))\leq d(D)-1$.
\end{prop}


\subsection{Subrings of given shape and counting}

\begin{lem} The extension $\varphi: R_{n,N,k+1}\to R_{n,N,k}$ is minimal. Furthermore, for any subring $R\subseteq R_{n,N,k+1}$ containing $p^kx^{n-1}$, the extension $R\to R/(p^kx^{n-1})$ is minimal.
\end{lem}

\begin{proof} In fact $(p^kx^{n-1})$ is the unique minimal ideal of $R_{n,N,k+1}$. The second follows as well (applying restrictions of minimal extensions, Proposition \ref{prop-7}).
\end{proof}

Here's the main use of our results on minimal extensions:

\begin{thm}\label{thm-39} Let $R\subseteq R_{n,N,k+1}$ subring.

\begin{enumerate}
\item If $p^{k}x^{n-1}\in \eme_R^2+pR$, there are no subrings of $R_{n,N,k+1}$ mapping isomorphically onto $\varphi(R)=S$.
\item Otherwise, the set of such subrings is parametrized by an affine space over $\F_p$ of dimension $d(S)=\dim_{\F_p}(\eme_S/(\eme_S^2+pS))$, hence there are $p^{d(S)}$ of them.
\end{enumerate}
\end{thm}

\begin{proof} This is Theorem \ref{main-thm} part 2, since the subrings involved have characteristic $p^N$ and residue field $\F_p$.
\end{proof}

\begin{cor} Suppose that $A\subseteq R_{n,N,k+1}$ and that $(n-1, k-1)\notin D(A)$. Then the number of subrings of $R_{n,N,k}$ mapping isomorphically onto $\varphi(A)$ is $p^{d(A)}$ where $d(A)\leq d(D(A))$.
\end{cor}

\begin{proof} Let $S=\varphi(A), R=\varphi^{-1}(A)$. Since $(n-1, k-1)\notin D(A)$, the map $A\to S$ is an isomorphism and applying Theorem \ref{thm-39}, the number of such subrings is $p^{d(S)}$ where $d(S)=\dim_{\F_p}(\eme_S/(\eme_S^2+pS))=\dim_{\F_p}(\eme_A/(\eme_A^2+pA))=d(A)$ and $d(A)\leq d(D(A))$.
\end{proof}


\subsubsection{Shape}

\begin{definition}  Let $D\subseteq \mathcal{M}_{n,N}$ a sub-partial-monoid. The collection of subrings of $R_{n,N,k}$ of shape $D$ is $\mathcal{S}_{n,N,k}(D)=\{ R\subseteq R_{n,N,k} \, \mid D(R)=D\}$, a partition of the set of subrings of $R_{n,N,k}$.
\end{definition}

We need a simple characterization of those $D$ that are the set of exponents of a subring.

\begin{prop}\label{prop-42} $D\subseteq \mathcal{M}_{n,N}$ is the set of exponents of a subring $R\subseteq R_{n,N,k}$ $\iff$ $D$ is a partial-sub-monoid of $\mathcal{M}_{n,N}$ containing $(p^i, 0)$ for all $0\leq i\leq N-1$.
\end{prop}

\begin{prop}  Let $D\subseteq \mathcal{M}_{n,N}$ a sub-partial-monoid, $R\subseteq R_{n,N,k+1}$ subring, $D(R)=D$.
\begin{enumerate}
\item If $(n-1, k)\in D$, the mapping $R\mapsto R/(p^kx^{n-1})$ induces a bijection of sets  $\mathcal{S}_{n,N,k+1}(D)\mapsto  \mathcal{S}_{n,N,k}(D\setminus \{(n-1, k)\})$
\item If $(n-1, k)\notin D$, the mapping $R\mapsto R/((p^kx^{n-1})\cap R)\cong R$, induces a mapping $\mathcal{S}_{n,N,k+1}(D)\mapsto  \mathcal{S}_{n,N,k}(D)$. Moreover, for $D_1\neq D_2$ sub-partial-monoids of $\mathcal{M}_{n,N}$ such that $(n-1,k)\notin D_1$, $(n-1,k)\notin D_2$, the images of $\mathcal{S}_{n+1}(D_1)$ and $\mathcal{S}_{n+1}(D_2)$ are disjoint, under the mapping just described.
\item If $(n-1, k)\notin D$, the mapping above $\mathcal{S}_{n,N,k+1}(D)\mapsto  \mathcal{S}_{n,N,k}(D)$ is surjective $\iff$ for all $B\in \mathcal{S}_{n,N,k}(D)$ the ring $R=\varphi^{-1}(B)$ possesses a subring $A$ mapping isomorphically to $B$ $\iff$ the kernel $Ker(R\to B)$ doesn't lie in $\eme_R^2+pR$ $\iff$ $d(R)=d(B)+1$.

\end{enumerate}
\end{prop}

\begin{proof} 
\begin{enumerate}
\item By Proposition \ref{prop-36}, the mapping is well defined, and so is the inverse mapping $B\mapsto \varphi^{-1}(B)$, $B\subseteq R_{n,N,k}$. It's immediate to see that the composition in both ways yields the identity on the sets $\mathcal{S}$.
\item By Proposition \ref{prop-36}, for $(n-1,k)\notin D$, the the mapping $R\mapsto R/((p^kx^{n-1})\cap R)\cong R$ induces an equality of sets $D(R)=D(R/((p^kx^{n-1})\cap R))$, hence both claims follow. 
\item The last two equivalences follow in the same way as Proposition \ref{prop-28} follows from Proposition \ref{prop-8}. For the first one, to have a surjective map $\mathcal{S}_{n,N,k}(D)\mapsto  \mathcal{S}_{n,N,k}(D)$ amounts to have that $R$ is not the only ring mapping to $B$, which by theory of minimal extensions (Theorem \ref{main-thm}), is the same as saying $R$ has a subring mapping isomorphically onto $B$.
\end{enumerate}
\end{proof}

We can proceed to estimate the number of subrings. For $D\subseteq \mathcal{M}_{n, N}$ a sub-partial-monoid containing $(p^i,0)$ for all $0\leq i\leq N-1$, let $\epsilon(D)$ be the following quantity recursively defined, with $d(D)$ being the minimal number of generators (so $d(\{0\})=0$)

  \[ \epsilon_{n, N, k}(D) = \begin{cases}
              0 &\text{if~} n=1 \\
               \epsilon_{n,N,k-1}(D\setminus \{(n-1,k-1)\})     & \text{if~} (n-1, k-1)\in D, n > 1\\
               d(D)-1+\epsilon_{n,N,k-1}(D)     & \text{if~} (n-1,k-1)\notin D, n > 1\\
            \end{cases} \]
           
\noindent Here's the main counting result:

\begin{thm}\label{thm-43} Let $D\subseteq \mathcal{M}_{n,N}$. Then $\#(\mathcal{S}_{n,N,k}(D))$ is at most $p^{\epsilon_{n,N,k}(D)}$.
\end{thm}

\begin{prop} For $D\subseteq \mathcal{M}_{n,N}$ sub-partial-monoid:
\begin{itemize}
\item If $(n-1, k)\in D$, $\#(\mathcal{S}_{n, N, k+1}(D)) = \#(\mathcal{S}_{n, N, k}(D\setminus \{(n-1,k)\}))$
\item If $(n-1, k)\notin D$, $\#(\mathcal{S}_{n,N,k+1}(D))\leq p^{d(D)-1}\#(\mathcal{S}_{n}(D))$ and furthermore, if for all $B\in \mathcal{S}_{n,N,k}(D)$, $d(\varphi^{-1}(B))=d(B)+1$, then $\#(\mathcal{S}_{n,N,k+1}(D))= p^{d(D)-1}\#(\mathcal{S}_{n,N,k}(D))$
\end{itemize}
\end{prop}

\begin{prop} Let $B\subseteq R_{n,N,k}$ subalgebra and $R=\varphi^{-1}(B)$. If $d(R)=d(D(R))-1$, then $d(B)=d(D(B))-1$.
\end{prop}

\begin{prop} With the same setup as above, assume $d(R)=d(B)+1$. If $d(B)=d(D(B))-1$ then $d(R)=d(D(R))-1$.
\end{prop}

\Addresses


\begin{thebibliography}{1}

\bibitem{AKKM} S. Atanasov, N. Kaplan, B. Krakoff, and J. Menzel, Counting finite index subrings of $\Z^n$ and $\Z[x]/(x^n)$, arxiv.org/abs/1609.06433

\bibitem{Atiyah} M.F. Atiyah, I.G. Macdonald, {\em Introduction to Commutative Algebra.} Addison--Wesley, Reading, MA, 1969.

\bibitem{Franco} Franco Munoz, Subrings of finite commutative rings.
  
\bibitem{Ganske} Ganske, G.; McDonald, B.R., Finite local rings. Rocky Mountain J. Math. 3 (1973), no. 4, 521-540.
    
\bibitem{MacDonald} MacDonald B. R., Finite rings with identity, M. Dekker (1974).

\bibitem{RS} J. C. Rosales, J. C and P. A. Garc\'{i}a-S\'{a}nchez, P. A. Numerical semigroups. Developments in Mathematics, 20. Springer, New York, 2009

\bibitem{SP} The Stacks Project {https://stacks.math.columbia.edu/}

\end{thebibliography}
\end{document}